\documentclass[apam]{degruyter-plain}   
\journalnameshort{Page count}
\usepackage{tikz-cd}

\renewcommand{\Im}{\mathop{\mathrm{Im}}}
\renewcommand{\Re}{\mathop{\mathrm{Re}}}

\newcommand{\Fhyp}{{}_1\mkern-2mu\hbox{$F_2$}}

\title{Operational calculus and integral transforms for
groups with finite propagation speed}
\headlinetitle{Functional calculus and integral transforms}

\lastnameone{Blower}
\firstnameone{Gordon}
\nameshortone{G.~Blower}
\addressone{Department of Mathematics and Statistics, Lancaster
University, Lancaster LA1 4YF}
\countryone{UK}
\emailone{g.blower@lancaster.ac.uk}

\lastnametwo{Doust}
\firstnametwo{Ian}
\nameshorttwo{I.~Doust}
\addresstwo{School of Mathematics and Statistics, UNSW Australia, Sydney NSW 2052}
\countrytwo{Australia}
\emailtwo{i.doust@unsw.edu.au}

\abstract{Let $A$ be the generator of a strongly continuous cosine family $(\cos (tA))_{t\in {\bf R}}$ on a complex Banach space $E$. The paper develops an operational calculus for integral transforms and functions of $A$ using the generalized harmonic analysis associated to certain hypergroups. It is shown that characters of hypergroups which have Laplace representations give rise to bounded operators on $E$. Examples include the Mellin transform and the Mehler--Fock transform. The paper uses
functional calculus for the cosine family $\cos( t\sqrt {\Delta})$ which is
associated with waves that travel at unit speed. The main results include an operational calculus theorem for Sturm--Liouville hypergroups with Laplace representation as well as analogues to the Kunze--Stein phenomenon in the hypergroup convolution setting.}

\keywords{Operator groups, multipliers, hypergroups}

\classification{Primary: 47A60; Secondary: 47D09}

\researchsupported{This research was partially supported by a Scheme~2 Grant from the London Mathematical Society.}

\acknowledgments{GB thanks the University of New South Wales for hospitality.}

\begin{document}

\section{Introduction}

Let $E$ be a separable complex Banach space and $\mathcal{L}(E)$ the algebra of bounded linear operators on $E$. Let $A$ a closed and densely defined linear operator in $E$. This paper presents a unified approach to the operational calculus
of functions $f(A)$  which is based upon integral transforms, including those in the following table.
\begin{table}[h]
\scalebox{0.85}{
\begin{tabular}{ccccc}
Transform& Characters& $L$    & $\phi_A (t)$&  Operations\\
\hline
\noalign{\vskip 2mm}
 Fourier & $\cos tx$& $-\frac{d^2}{dx^2} $& $\cos tA$ & cosine \\
Mellin  & $x^{it}$& $-(x\frac{d}{dx} )^2$ & $A^{it}$& Riesz potentials \\
Hankel & $x^{-\nu}J_\nu (\lambda x)$& $-\frac{d^2}{dx^2} -\frac{2\nu+1}{x}\frac{d}{dx}$ & $t^{-\nu }J_\nu (tA)$& Bessel \\
Mehler & $P_{i\lambda-(1/2)}^{0}(\cosh x)$& $-\frac{d^2}{dx^2} - \mathrm{coth} \, x\, \frac{d}{dx }$& ${U_{1/2}(\cos (tA)})$& Legendre
\end{tabular}
}
\end{table}

Associated to the differential operators $L$ that appear in this table there is a convolution $\ast$ defined initially on point masses $\varepsilon_x$ on ${\bf X} = [0,\infty)$ such that the convolution $\varepsilon_x \ast \varepsilon_y$ is a probability measure on ${\bf X}$. This convolution determines a hypergroup structure denoted $({\bf X},\ast)$. The characters $\phi$ of this hypergroup are multiplicative in the sense that they satisfy $\int_{\bf X} \phi(t) (\varepsilon_x \ast \varepsilon_y)(dt) = \phi(x) \phi(y)$. Working with the character space ${\hat {\bf X}}$ allows us to use generalized harmonic analysis to transfer estimates for $\sqrt{L}$ to $A$. (We refer the reader to \cite{Gigante:2001, Coifman:1976} for related transference methods.)
In Section~\ref{HypGps} we introduce the main facts from the theory of hypergroup structures on ${\bf X}$ that we shall need.

In Section~\ref{Sec:Mellin} we begin by investigating a classical situation regarding operators $A$ which admit bounded imaginary powers $A^{is}$ and a functional calculus derived from the Mellin transform. In the context of this work, the Mellin transform can be viewed as the generalized Fourier transform determined by a certain natural hypergroup structure on ${\bf X}$ and the imaginary powers of $A$ are just the values of $\phi(A)$ for $\phi$ in the character space of this hypergroup.

The remaining part of the paper aims to make formal use of the hypergroup Fourier transform formula
  \begin{equation}\label{EqGFT}
     {\hat f}(\phi) = \int_{\bf X} f(x) \phi(x) \, m(dx), \qquad (\phi \in {\hat {\bf X}},\ f \in L^1(m))
  \end{equation}
to define ${\hat f}(A)$. To do this, one needs to find a suitable way of replacing the scalar-valued $\phi(x)$ term with an operator-valued quantity $\phi_A(x)$. Here we make use of the fact that for certain hypergroups $({\bf X},\ast)$, the bounded multiplicative maps on ${\bf X}$ can be naturally parameterized as $\{\phi_\lambda\}$ for $\lambda$ in a subset of the complex plane, and furthermore, that for all $x \in {\bf X}$ the function $h_x(\lambda) = \phi_\lambda(x)$ is bounded and analytic on a suitable domain. Indeed these maps have a `Laplace representation' in terms of a family of bounded positive measures $\tau_x$,
  \begin{equation}\label{EqChRep}
      \phi_\lambda(x) = h_x(\lambda) = \int_{-x}^x \cos(\lambda t) \, \tau_x(dt), \qquad (x \in {\bf X}).
  \end{equation}
To make use of this representation to define $\phi_A(x) = h_x(A)$, one needs a satisfactory interpretation of, and bounds for, the family of operators $\{\cos(tA)\}_{t \in {\bf R}}$, as well as suitable bounds concerning the representation measures $\tau_x$. 

Cosine families of operators have a well-developed theory. Formally, a cosine family on $E$ is a strongly continuous family $\{C(t)\}_{t \in {\bf R}}$ of bounded operators on $E$ such that $C(s-t)+C(s+t) = 2C(s)C(t)$ and $C(0) = I$. Such a family admits a closed densely defined infinitesimal generator $A$ and one naturally writes $\cos(tA)$ for $C(t)$.  Cosine families arise in describing the solutions of well-posed $L^2$ Cauchy problems of the form
    \[ \frac{\partial^2w}{\partial t^2} = -A^2w, \qquad
  w(0)=u, \qquad \frac{\partial w}{\partial t}(0)=0
 \]
with initial datum $u\in L^2$. 
In classical situations, these systems admit wave solutions which propagate at a fixed finite speed.
We refer the reader to \cite{Travis:1978}, \cite[p. 118]{Goldstein:1985} or \cite{CheegerGT:1982} for further details.

Given a cosine family $\{\cos(tA)\}_{t \in {\bf R}}$, various authors (see, for example, \cite{CheegerGT:1982} or \cite{Taylor:1989}) have used this to use this to define
     \begin{equation}\label{Eq1.1}
     f(A)={\frac{1}{2\pi}}\int_{-\infty}^\infty {\cal F}f (t  ) \cos (t  A)\, dt  
     \end{equation}
where ${\cal F}f(t  )=\int_{-\infty}^\infty f(x)e^{-ixt }dx$ and $f$ is an even function in $C_c^\infty ({\bf R})$. Such an approach works well if, for example, the cosine family is uniformly bounded, but in general such familes are not so well-behaved. Even in the case that $E$ is an $L^p$ space
\begin{enumerate}
   \item[(i)] $\Vert  \cos (t  A)\Vert_{{\cal L}(L^2)}$ can grow exponentially with $\vert t \vert$ (see \cite[p. 118]{Goldstein:1985});
   \item[(ii)]  $\cos (t  A)$ can be unbounded as an operator on $L^p$ for $2<p<\infty$.
\end{enumerate}

In Section~\ref{Sec3} we give general conditions on $({\bf X},\ast)$ and $\{\cos(tA)\}$ which ensures even in the case that 
$\Vert  \cos (t  A)\Vert_{{\cal L}(L^2)}$ grows exponentially,  the family of  operators $\{\phi_A(x)\}$ is uniformly bounded and hence we can use (\ref{EqGFT}) to show that ${\hat f}(A)$ is bounded for all $f \in L^1(m)$. In Section~\ref{Sec:SLH} we show that certain Sturm--Liouville hypergroups associated to a differential operator $L$  do indeed have the desired properties. 

Several standard integral transforms appear from appropriate choices of hypergroup structure on ${\bf X}$. In Section~\ref{Sec:MF} we look at the hypergroup structure associated to the operator
  \[ L\phi(x) = -\phi''(x) - \coth x\, \phi'(x), \qquad (x \ge 0) \]
which generates the Mehler--Fock transform of order zero. In this setting, the operators $\phi_A(x)$ arise as fractional integrals of the cosine family. In the final section we show that the hypergroups associated to naturally occuring Laplace operators on certain Riemannian manifolds have the required properties for the earlier theory to apply.

For a locally compact group $G$, the space $L^1(G)$ acts boundedly on 
$L^2(G)$ by left-convolution. That is, if $f\in L^1(G)$ then $\Lambda_f: g\mapsto f\ast g$ is a bounded operator on $L^2(G)$. In general, this result does not extend to $f\in L^p(G)$ for $p > 1$. The Kunze--Stein phenomenon refers to the fact that for certain Lie groups, most classically for $G = \mathrm{SL}(2, {\bf C})$, for $1\leq p<2$ one does obtain a bound of the form 
   \[ \Vert f\ast g\Vert_{L^2(G)}\leq C_p \Vert f\Vert_{L^p(G)} \Vert g\Vert_{L^2(G)};\]
see \cite[p. 52]{Coifman:1976}. Thus the representation $\Lambda :(L^1(G),\ast )\rightarrow {\cal L}(L^2(G))$ 
$:f\mapsto \Lambda_f$ extends to a bounded linear map $\Lambda :L^p(G)\rightarrow {\cal L}(L^2(G))$.

Our main results Theorems \ref{Thm3.7} and \ref{Thm4.3} are analogues of this Kunze--Stein phenomenon. Indeed the classical case of $G = \mathrm{SL}(2, {\bf C})$ contains much of the hypergroup architecture that we explore in this paper. As is discussed in \cite{Jewett:1975},
$\mathrm{SL}(2, {\bf C})$ has a maximal compact
subgroup $K={\hbox{SU}}(2, {\bf C})$ such that $K\times K$ acts upon $G$ via
$(h,k):g\mapsto h^{-1}gk$ for $h,k\in K$ and $g\in G$, producing a space
of orbits $G//K=\{ KgK: g\in G\}$. The double coset space $G//K$ inherits the structure of a
commutative hypergroup modelled on ${\bf X}=[0, \infty )$ and as for the Sturm--Liouville hypergroups, we obtain representations linked to eigenfunctions of a differential operator on $(0, \infty)$. 
The reader is referred to Chapter~10 of \cite{Coifman:1976} for further details.

The functional calculus maps defined above factor through the Banach algebras $(L^1({\bf X},m),\ast)$.
In Theorem \ref{Thm3.7}, we produce a family of hypergroup representations
$\Phi: (L^1({\bf X},m), \ast )\rightarrow {\cal L}(E)$ that automatically extend to $\Phi : L^p({\bf X}, m)\rightarrow {\cal L}(E)$ for $1\leq p<2$. In Theorem \ref{Thm4.3} we obtain a version of this abstract theorem which applies to differential operators $L$ on $(0, \infty )$, as in the double coset hypergroup ${\bf X}=G//K$. We show that the space of bounded and multiplicative functions $\varphi_\lambda :({\bf X},\ast )\rightarrow {\bf C}$ is a strip $\{\lambda\in {\bf C}: \vert \Im \lambda\vert <\omega_0\}$, where $\omega_0>0$ is determined by $L$. 
The proof involves functional calculus for the cosine families and the Laplace representation, and was suggested by the results in \cite[p. 42]{CheegerGT:1982}.

Before progressing further, we shall fix some notation.
For $\omega >0$ we let $\Sigma_\omega$
denote the strip $\{ z\in {\bf C}: \vert \Im z\vert <\omega \}$ and $i\Sigma_\omega$ the corresponding vertical strip.
For $0<\theta <\pi$, we introduce the open sector
$S_\theta^0= \{z\in {\bf C}\setminus \{ 0\}: \vert \arg z\vert <\theta \}$
and its reflection $-S_\theta^0=\{ z:-z\in S_\theta^0\}$.
An important idea
is to work with holomorphic functions on `Venturi' regions;
that is, those of the form
  \[ V_{\theta,\omega }=\Sigma_\omega \cup S_\theta^0
\cup (-S_\theta^0).\]
Likewise, $iV_{\theta,\omega }$ will denote the corresponding Venturi region with vertical axis. As usual, $H^\infty(S)$ will denote the Banach algebra of bounded analytic functions on an open subset $S$ of the complex plane.

\section{Hypergroups on $[0,\infty)$}\label{HypGps}

In this section we introduce the general formalism of hypergroups with base space $[0,\infty)$. % and an associated operational calculus.  
A full account of harmonic analysis in the hypergroup context may be found in \cite{Bloom:1994},
\cite{Jewett:1975} or \cite{Zeuner:1989}.

Let ${\bf X}$ denote the half-line $[0,\infty)$, and $C_c({\bf X} )$ the space of compactly supported continuous functions $f:{\bf X} \rightarrow {\bf C}$.
The set $M^b({\bf X})$ of bounded Radon measures on ${\bf X}$ with the weak topology forms
a complex vector space. When equipped with a suitable associative multiplication
or `generalized convolution' operation $\ast$ on $M^b({\bf X})$, this
convolution measure algebra is called a hypergroup or `convo'. We shall usually denote this as $({\bf X},\ast)$ although one needs to remember that the
operations are defined on $M^b({\bf X} )$ rather than the underlying
base space ${\bf X}$.

Denote the Dirac point mass at $x$ by $\varepsilon_x \in
M^b({\bf X})$. It is a hypergroup axiom that for all $x,y \in
{\bf X}$, $\varepsilon_x \ast \varepsilon_y$ is a
compactly supported probability measure. The action of $\ast$ in a
hypergroup is in fact completely determined by the convolutions $\varepsilon_x
\ast \varepsilon_y$. When the base space is
${\bf X} =[0, \infty )$, the convolution $\ast$ is necessarily commutative, $\varepsilon_0$
is a multiplicative identity element. In general, hypergroups admit an involution map $x\mapsto x^{-}$. 
For $x \in {\bf X}$, the left translation operator $\Lambda_x$ is defined,
initially on $C_c({\bf X})$ by
  \[ \Lambda_x f(y)
    = \int_{{\bf X}} f(t)\, (\varepsilon_x\ast \varepsilon_y)(dt)
                \qquad (x,y\in {\bf X}).\]
It is traditional and useful to write $\Lambda_x f(y)$ as $f(x \ast y)$ (although this is not in fact defining an operation on $X$).
Since $\ast$ is commutative, there exists an
essentially unique Haar measure on ${\bf X}$; that is, a nontrivial positive invariant measure $m$ on
$[0,\infty)$ satisfying
  \begin{equation*}\label{Eq3.1}
  \int_{{\bf X}} \Lambda _xf(y)\, m(dy)
      =\int_{{\bf X}} f(y)\, m(dy)\qquad (x\in {\bf X}).
  \end{equation*}
for all $f \in C_c({\bf X})$; see \cite[Section~1.3]{Bloom:1994}. This allows us to define a (commutative) convolution between two functions $f,g \in C_c({\bf X})$ by
  \[ (f \ast g)(x) = \int_{\bf X} f(y)\, \Lambda_xg(y) \, m(dy)
                    = \int_{\bf X} f(y)\, g(x \ast y) \, m(dy). \]
This map extends to $L^1(m)=L^1({\bf X}, m)$ and makes $(L^1(m),\ast)$ into a commutative Banach algebra. One often writes 
the convolution operation as $\Lambda_fg=f\ast g$ for $f,g\in L^1(m)$.

\begin{definition}\label{Defn3.1} 
\begin{enumerate}
   \item A continuous function $\phi: {\bf X} \to {\bf C}$ is said to be \textit{multiplicative} if $\phi(x \ast y) = \phi(x) \phi(y)$ for 
all $x,y \in {\bf X}$ and $\phi (z)\neq 0$ for some $z\in {\bf X}$. 
   \item A \textit{character} on the hypergroup ${\bf X}$ is a bounded and multiplicative function $\phi$ such that $\phi(x^{-}) = \overline{\phi(x)}$  and $\phi(0) = 1$. The \textit{character space} $\hat{\bf X}$ is the set of all characters on ${\bf X}$.
\end{enumerate}
\end{definition}

When $X = [0,\infty)$ the involution is always the identity $x^{-}=x$, and the condition that $\phi(x^{-}) = \overline{\phi(x)}$ is equivalent to the condition that $\phi(x) \in {\bf R}$ by \cite[Theorem 3.4.2]{Bloom:1994} and this simplifies some of the definitions below. In section 3, we use multiplicative functions which are bounded but not characters. In the cases of interest to us in this paper, Definition~\ref{Def4.1}, the hypergroup convolution is associated with a differential operator and the multiplicative functions are eigenfunctions of this operator. Indeed, the set of bounded and multiplicative functions $\phi_\lambda$ can be naturally parametrized by a domain $S_{\bf X} \subseteq {\bf C}$. This occurs, in particular, for Sturm--Liouville hypergroups, in which case $\lambda$ is a spectral parameter as in \cite{Bloom:1994},\cite{Chebli:1974},\cite{Chebli:1979}  and \cite{Trimeche:1981}.  
The character space $\hat{\bf X}$ is always sufficiently large in our context to enable one to do
harmonic analysis.
We can define the Fourier transform of $f \in
L^1({\bf X} ;m)$ by setting
  \begin{equation} \label{Eq3.2}
   {\hat f}(\phi) = \int_{{\bf X}} f(x) \phi(x)  \, m(dx), \qquad (\phi \in \hat{\bf X}).
  \end{equation}
In the case that $\hat{\bf X} \subseteq \{\phi_\lambda \,:\, \lambda \in S_{\bf X}\}$ we shall write ${\hat f}(\lambda)$ rather than ${\hat f}(\phi_\lambda)$ and we can extend ${\hat f}$ to be a function of the complex variable $\lambda$.

By a theorem of Levitan \cite{Jewett:1975}, there exists a unique
Plancherel measure $\pi_0$ supported on a closed subset ${\bf S}$ of
$\hat {{\bf X}}$ such that $f\mapsto \hat f$ for $f\in L^2(m)\cap L^1(m)$ extends to a unitary isomorphism $L^2(m)\rightarrow L^2(\pi_0)$. By \cite[Theorem 2.3.19]{Bloom:1994} or \cite{Voit:1988}, there exists a unique
positive character
$\phi_0\in {\bf S}$, and $\phi_0$ can be different from the trivial character ${\bf I}$. Indeed, this enables us to deal with unbounded cosine families, as in Proposition~\ref{Prop3.5} below.

\begin{definition}\label{Defn3.2} 
A hypergroup $({\bf X},\ast)$ is said to have a {\sl Laplace representation} if $(a,b)\subseteq {\bf S}$ for some $0<a<b$, and for every $x \ge 0$, there exists
 a positive Radon measure $\tau_x$
on $[-x,x]$ such that $\tau_x([-x,x])=\phi_0(x)$ and for every character $\phi_\lambda$ in ${\bf S}$
  \begin{equation}\label{Eq3.3}
    \phi_\lambda (x)=\int_{-x}^x\cos (\lambda t)\tau_x(dt).
  \end{equation}
  The integral is taken over $[-x,x]$, and includes any point masses at $\pm x$.
\end{definition}

The Sturm--Liouville hypergroups that we shall consider in Section~\ref{Sec:SLH} all admit a Laplace representation. 
For the rest of this section therefore, we assume that $({\bf X},\ast)$ has a Laplace representation. Note that the right-hand side of (\ref{Eq3.3}) converges for all $\lambda \in {\bf C}$ and all $x  \ge 0$, so the Laplace representation allows us to move from the character space to a larger subset of ${\bf C}$.

\begin{lemma}\label{3.4(i)}
Let ${\bf X}$ be as in Definition \ref{Defn3.2}. Suppose that there exist $M_0, \omega_0>0$ such that
  \begin{equation}\label{Eq3.4}
     \int_{-x}^x \cosh (\omega_0 t)\, \tau_x(dt)\leq M_0\qquad (x\geq 0).
   \end{equation}
Then
\begin{enumerate}
 \item for all $\lambda \in \Sigma_{\omega_0}$ the function $\phi_\lambda: {\bf X} \to {\bf C}$,
   \[ \phi_\lambda(x) = \int_{-x}^x\cos (\lambda t)\, \tau_x(dt) \qquad (x\geq 0)\]
  is bounded and multiplicative;
 \item for all $x \in {\bf X}$, the map $h_x: \lambda \mapsto \phi_\lambda(x)$ is in $H^\infty(\Sigma_{\omega_0})$;
  \item ${\bf R}\cup [-i\omega_0, i\omega_0]$ is contained in $\hat {\bf X}$;
 \item the Fourier transform $f \mapsto {\hat f}$ is bounded $L^1(m) \to H^\infty(\Sigma_{\omega_0})$.
\end{enumerate}
\end{lemma}

\begin{proof} If $\lambda = u+iv \in \Sigma_{\omega_0}$ then $|\cos(\lambda t)| \le \cosh(vt) \le \cosh(\omega_0 t)$ which shows that $|\phi_\lambda(x)| \le M_0$.
From this inequality and Morera's theorem, it also follows that $h_x\in H^\infty(\Sigma_{\omega_0})$.

Now $\phi_0(0)=1$ since $\phi_0\in \hat {\bf X}$, so $\phi_\lambda (0)=1$ for all $\lambda\in {\bf C}$. By Definition \ref{Defn3.2}, $\phi_\lambda (x)$ is multiplicative for all $\lambda \in (a,b)$ and by analytic continuation for all $\lambda\in \Sigma_{\omega_0}$. This completes the proof of (i) and (ii).

(iii) It is clear from the definition of $\phi_\lambda$ that if $\lambda$ real or purely imaginary then 
$ \phi_\lambda(x) \in {\bf R}$.  Hence $\phi_\lambda$ is a character of ${\bf X}$ for all $\lambda\in {\bf R}\cup [-i\omega_0 ,i\omega_0]$.

(iv) Finally, we have $\vert \hat f(\lambda )\vert \leq \int_{0}^\infty M_0\vert f(x)\vert\, m(dx)$ for all $f\in L^1(m)$, so 
(iv) follows from (i) by convexity. 
\end{proof}

\section{An operational calculus from the Mellin transform }\label{Sec:Mellin}

A canonical example of a hypergroup structure on $(0,\infty )$ is given by the convolution $\varepsilon_x\ast\varepsilon_y=\varepsilon_{xy}$. In this case the invariant measure is $dx/x$ and bounded characters are $\phi_\tau (x)=x^{i\tau}$ ($\tau \in {\bf R}$). The Fourier transform in this case is thus
  \[ {\hat f}(\phi_\tau) = \int_0^\infty f(x) x^{i\tau} \, \frac{dx}{x} \]
which is traditionally written as $f^*(i\tau)$,  the Mellin transform of $f$ evaluated at $i\tau$. 

If $A$ is a sectorial operator on a Hilbert space such that for some
$M, \omega_1\geq 0$,
$\Vert A^{i \tau} \Vert_{{\cal L}(H)}\leq Me^{\omega_1 \tau}$ for all $\tau \in {\bf R}$, then $A$ has a
bounded $H^\infty (\Sigma_{\omega_0})$ functional calculus on $H$ for
all $\omega_0>\omega_1$. 
Example 5.2 in \cite{Cowling:1996}  shows that this results does not extend from $H=L^2({\bf R})$ to $L^p({\bf R})$ for $p\neq 2$. To address this issue, we provide an operational calculus results based on the Mellin transform. (The use of the Mellin transform is of course not novel: see, for example,  \cite{Blower:1996}, \cite{Cowling:1996}, \cite{McIntosh:1986} or \cite{Uiterdijk:2000}).

We recall a Mellin transform theorem from \cite[p.~273]{Sneddon:1972}. Let $f^*$ be holomorphic on
$i\Sigma_\alpha$ and suppose that $e^{\vert \tau\vert\mu} f^*(\sigma +i\tau )\rightarrow 0$ uniformly on $i\Sigma_{\alpha -\varepsilon}$ as $\tau\rightarrow\pm  \infty$ for some $\alpha >\varepsilon >0$ and $\mu \leq \pi$. Then $f^*(s)$ is the Mellin transform of
  \begin{equation*}\label{eq2.1}
  f(z)=\frac{1}{2\pi i} \int_{\sigma -i\infty}^{\sigma +i\infty} z^{-s} f^*(s)\, ds \qquad (0 < z < \infty).
  \end{equation*}

\begin{proposition}\label{Prop2.1}
 Suppose that $1<p<\infty$ and that $E$ is a closed linear subspace of $L^p(\Omega,\mu)$ for some measure space $(\Omega,\mu)$. Suppose also that
\begin{enumerate}
 \item $A$ is a one-to-one operator in $E$ such that $(A^{i\tau})_{\tau\in {\bf R}}$ is a $C_0$ group of operators on $E$ and $\Vert A^{i\tau}\Vert_{{\cal L}(E)}\leq C$ for all $\tau\in {\bf R}$;
 \item $f^* \in H^\infty(iV_{\theta ,\omega })$ for some $\theta ,\omega >0$, that $f^*$ is continuous on the closure of $iV_{\theta ,\omega}$ and $f^*(s )\rightarrow 0$ as $\vert s\vert\rightarrow \infty$, uniformly with respect to $\arg s$ for $s\in iV_{\theta ,\omega}$.
\end{enumerate}
Then
  \begin{equation}\label{Eq2.2}
  f(A)=\frac{1}{2\pi}\int_{-\infty}^\infty A^{-i\tau} f^*(i\tau )\,d\tau
  \end{equation}
defines a bounded linear operator on $E$.
\end{proposition}

\begin{proof} By Cauchy's estimates, there exists $C_{\theta ,\omega}>0$ such that
   \begin{equation}\label{Eq2.3}\vert f^*(i\tau )\vert +\vert \tau \vert \Bigl\vert \frac{df^*}{d\tau}(i\tau )\Bigr\vert \leq C_{\theta ,\omega}
     \qquad (\tau \in {\bf R}),
   \end{equation} 
hence $f^*(i\tau)$ defines a Fourier multiplier on $L^p({\bf R})$ as in Ste{\v c}kin's Theorem. By the Berkson--Gillespie transference theorem \cite{Berkson:1987}, the integral (\ref{Eq2.2}) defines a bounded linear operator on $E$.
\end{proof}

Next we extend the result to groups of exponential growth. We note here the relatively standard (and easily proven) fact about analytic continuation of a function on $(0,\infty)$ defined using the Mellin transform.

\begin{lemma}\label{Lem2.2}
Suppose that $0 < \phi < \omega$. If 
$f^*(s) \cos(\omega s)$ belongs to $H^\infty(iV_{\theta,\alpha})$ for some $0 < \theta < \frac{\pi}{2}$ and $\alpha > 0$,
then
   \begin{equation}\label{Eq2.4}
    f(z) = \frac{1}{2\pi} \int_{-\infty}^\infty z^{-i \tau} f^*(i\tau) \, d\tau 
   \end{equation}
belongs to $H^\infty(S_\phi^0)$.
\end{lemma}

\begin{proof} Note that for $\tau \in {\bf R}$,  $|f^*(i\tau) \cos(i \omega \tau)| = |f^*(i\tau)| \cosh(\omega \tau)$ and hence $f^*(i\tau) = O(e^{-\omega |\tau|})$ as $\tau \to \pm \infty$. On the other hand if $z \in S_\phi^0$, then $|z^{-i\tau}| \le e^{\phi |\tau|}$. It follows therefore that the integral (\ref{Eq2.4}) converges absolutely. The analyticity of $f$ is standard.
\end{proof}

\begin{proposition}\label{Prop2.2}
Suppose that $\omega,\alpha > 0$, that $0 < \theta < \frac{\pi}{2}$, that $0 < \omega_0 < \omega$, and that
\begin{enumerate}
 \item $f^*(s) \cos(\omega s)$ belongs to $H^\infty(iV_{\theta,\alpha})$; and
 \item $(A^{i\tau})_{\tau\in {\bf R}}$ is a $C_0$ group on a Banach space $E$ such that $\Vert A^{i\tau}\Vert_{{\cal L}(E)} \leq Ce^{\omega_0\vert \tau\vert}$ for all $\tau\in {\bf R}$.
\end{enumerate}
Then $f(A)$, defined by (\ref{Eq2.2}), is a bounded linear operator on $E$.
\end{proposition}

\begin{proof}
The absolute convergence of the integral (\ref{Eq2.2}) follows easily from (ii) and the bounds in the proof of Proposition~\ref{Prop2.1}. Since the integrand is strongly continuous, the integral for $f(A)$ converges.
\end{proof}

\begin{remark}\label{Rem2.3}
An example
in \cite{Cowling:1996} and \cite{Boyadzhiev:1994} shows that for each $q\neq 2$ and $0<\theta <\pi$, there exists
$f\in H^\infty (\Sigma_\theta )$ that is not a bounded Fourier multiplier on
$L^q({\bf R})$. 
\end{remark}

\begin{example}\label{Ex2.4}
We consider a specific example at the margins of the scope of Proposition~\ref{Prop2.1}. Let $J_0$ be Bessel's function of the first kind of order zero, and for $x > 0$ let $g(x)=\sqrt{x}J_0(x)$. By \cite[p.522]{Sneddon:1972} $g$ has Mellin transform
  \[ g^*(s)=\frac{2^{s-1/2}}{\pi} \sin \pi\Bigl( \frac{s}{2}+\frac{1}{4}\Bigr) \, \Gamma \Bigl( \frac{s}{2}+\frac{1}{4}\Bigr)^2 \]
which is holomorphic for $s\in i\Sigma_{\alpha }$ for $0<\alpha <1/2$ and of polynomial growth as $is\to \infty$.
For $N > 0$ consider the functions (as in \cite{Erdelyi:1954v1})
  \[ h_N(x)=\frac{2N}{\pi} \frac{x^N}{1+x^{2N}},
  \quad
  h_N^*(s)=\sec  \Bigl( \frac{\pi s}{2N}\Bigr).\]
Then $h_N^*(is)\in H^\infty (V_{\theta,\beta})$ for $0<\beta <N$ and $0<\theta <\pi /2$, and $h_N^*(s)\rightarrow 1$ as $N\rightarrow\infty$, uniformly on compact subsets of ${\bf C}$.
The Mellin convolution $f_N=g\ast h_N$ from \cite[p.~276]{Sneddon:1972} has Mellin transform $f_N^*(s)=g^*(s)h^*_N(s)$ which is bounded and holomorphic for $s\in i\Sigma_\alpha$, for
$1/2<N<\infty$, although $f_N^*$ becomes unbounded whenever we extend $i\Sigma_\alpha$ to $iV_{\theta,\alpha}$ for $\theta >0$; so Proposition \ref{Prop2.1} (ii) does not apply directly. Nevertheless, by invoking standard asymptotic estimates  on the $\Gamma$ function from \cite[p.~279]{Whittaker:1927}, one can check  that (\ref{Eq2.3}) holds for $f_N^*$. We deduce that the conclusion of Proposition~\ref{Prop2.1} holds for $f_N$.
The $f_N$ can be computed in terms of standard special functions. In particular, using the table of Stieltjes transforms in \cite[14.3(6)]{Erdelyi:1954v2}, we can compute, in terms of the hypergeometric function $\Fhyp$, % $ {}_1F_{2}$, 
  \begin{align*}
  f_1(x)&=\int_0^\infty \frac{(x/y)}{1+(x/y)^2 } \sqrt {y} J_0(y)\frac{dy}{y } \\
    &={\frac{1}{2i}}\int_0^\infty\Bigl( {\frac{1}{y-ix}}- {\frac{1}{y+ix}}\Bigr) \sqrt{y}J_0(y)\, dy\\
    &= {\frac{\pi \sqrt{x} J_0(ix)}{\sqrt{2}}}+ {\frac{\Gamma (-\frac{1}{4})}{2^{3/2}\Gamma (\frac{5}{4})}} \, 
                \Fhyp   \Bigl(1;\tfrac{5}{4},\tfrac{5}{4};\tfrac{x^2}{4}\Bigr)x.
  \end{align*}
\end{example}

\section{An operational calculus from hypergroup convolution}\label{Sec3}

In this section we shall suppose that the operator $A$ generates a strongly continuous cosine family $(\cos (tA))_{t\in {\bf R}}$ on $E$, and that $({\bf X},\ast)$ is a hypergroup which admits a Laplace representation for its characters $\phi_\lambda$ as given in Definition~\ref{Defn3.2}.

In this setting we 
define the family of bounded linear operators \{$\phi_A(x)\}_{x \ge 0}$ on $E$ by the strong operator convergent integrals
  \begin{equation}\label{Eq3.5}
  \phi_A(x)=\int_{ -x}^{x} \cos (At)\, \tau_x(dt)\qquad (x\geq 0).
  \end{equation}
%\end{definition}

Note that one can easily verify that in simple situations (such as if $A$ is a normal matrix), $\phi_A(x) = h_x(A)$, where $h_x(\lambda) = \phi_\lambda(x)$ and  the right-hand side is interpreted via the usual Riesz functional calculus.
We now seek to define ${\hat f}(A)$ for suitable functions $f$ via the hypergroup Fourier transform by writing it as an integral of these operators. 

\begin{proposition}\label{Prop3.5}
Let $({\bf X}, *)$ have a Laplace representation satisfying (\ref{Eq3.4})
and suppose that $A$ generates a strongly continuous cosine family on $E$ satisfying
   \begin{equation}\label{Eq3.6} 
      \Vert \cos (tA)\Vert_{{\cal L}(E)}\leq \kappa\cosh (t\omega_0)\qquad (t\geq 0). 
    \end{equation}
Then
\begin{enumerate}
 \item $(\phi_A(x))_{x>0}$ is a uniformly bounded family of operators;
 \item for all $f \in L^1(m)$, the following integral converges in the strong operator sense
    \begin{equation}\label{Eq3.7}
      T_A  (f)=\int_0^\infty f(x)\phi_A(x) \,m(dx)
    \end{equation}
  and defines a bounded linear operator on $E$;
  \item for $f,g\in L^1(m)$, $T_A(f\ast g)=T_A (f)T_A (g)$, and so the map $T_A: L^1(m) \to {\cal L}(E)$ is an algebra homomorphism.
\end{enumerate}
\end{proposition}

\begin{proof}
(i)  We observe that by convexity $\phi_A(x)$ is
a bounded linear operator on $E$, and $\Vert \phi_A(x)\Vert_{{\cal L}(E)}\leq
\kappa M_0$.

(ii) Conclusion (ii) follows from (i) by convexity.

(iii) From the identity $\phi_\lambda (x\ast y)=\phi_\lambda (x)\phi_\lambda (y)$ and the Laplace representation (\ref{Eq3.3}), we have
   \begin{multline}\label{Eq3.8}
  \int \cos\lambda u \int \tau_z(du) (\varepsilon_x\ast\varepsilon_y)(dz) \\
   = {\frac{1}{2}}\iint \cos\lambda (t+s) \tau_x (dt)\tau_y(ds)  
       +{\frac{1}{2}}\iint \cos\lambda (t-s) \tau_x (dt)\tau_y(ds).
       \end{multline} 
So by the addition rule $\cos ((t-s)A)+\cos ((t+s)A)=2\cos (tA)\cos (sA)$ for the cosine family, the identity
$\phi_A(x\ast y)=\phi_A(x)\phi_A(y)$ follows unambiguously when one formally replaces $\lambda$ by $A$ in (\ref{Eq3.8}). We have
  \[ \int_0^\infty \int_0^\infty\phi_\lambda (x\ast y)f(x)g(y)\,m(dx)\,m(dy)
    =\int_0^\infty \phi_\lambda (z)(f\ast g)(z)\,m(dz)
   \]
by a standard identity \cite[6.1F]{Jewett:1975}, so 
\[ \int_0^\infty \int_0^\infty\phi_A (x\ast y)f(x)g(y) \, \,m(dx)\,m(dy)
    =\int_0^\infty \phi_A (z)(f\ast g)(z)\, m(dz)
   \]
 so we can express the left-hand side as a product of operators
\[\int_0^\infty\phi_A (x)f(x) \,m(dx) \int_0^\infty \phi_A(y)g(y)\ m(dy)
    =\int_0^\infty \phi_A (z)(f\ast g)(z)\, m(dz)
   \]
so that $f\mapsto T_A(f)$ is multiplicative.
\end{proof}

\begin{remark}\label{Rem3.6}
We interpret $T_A(f)$ in the above theorem as ${\hat f}(A)$.
The map $T_A: L^1({\bf X},m) \to {\cal L}(E)$ is a Banach algebra homomorphism  which generates a functional calculus map $\Phi_A (\psi) = \psi(A) = T_A \circ \mathcal{F}_{\bf X}^{-1}(\psi)$ defined for $\psi$ in the range $\mathcal{A}$ of the Fourier transform $\mathcal{F}_{\bf X}$.

\[
\begin{tikzcd}
(L^1({\bf X},m),\ast)  \arrow[leftrightarrow]{r}{\mathcal{F}_{\bf X}} \dar{T_A}
                   & {\mathcal A} \subseteq H^\infty(\Sigma_{\omega_0}) \dlar[dashed]{\Phi_A} \\
{\cal L}(E) &
\end{tikzcd}
\]
It is natural to ask whether the map $\Phi_A$ extends to a bounded algebra homomorphism $H^\infty(\Sigma_{\omega_0}) \to {\cal L}(E)$.
\end{remark}

According to \cite[Section 2.5.6]{Bloom:1994} and \cite{Voit:1988}, a noncompact commutative hypergroup has the Kunze--Stein property of order $p>1$ if $\Lambda_f$ gives a bounded linear operator on $L^2(m)$ for all $f\in L^p(m)$. In the following result, we refine this result by extending $\hat f$ to give an analytic function on a strip containing ${\bf S}$ and obtain an operational calculus.  
  To accommodate $p > 1$ we rescale the speed of $\cos(tA)$ to $\cos(\alpha tA)$ with $0 < \alpha <1$. 
Since our hypergroups are noncompact and commutative, \cite[Theorem~7.2B]{Jewett:1975} and \cite[Theorem ~2.5.6]{Bloom:1994} say that $\phi_0$ is not in $L^\nu (m)$ for $1\leq \nu \leq 2$. The following result therefore includes the optimal range of exponents.

\begin{theorem}\label{Thm3.7} Let $({\bf X}, *)$ have a Laplace representation satisfying (\ref{Eq3.4}) and suppose that $A$ generates a strongly continuous cosine family on $E$ satisfying (\ref{Eq3.6}). Suppose further that $\phi_0 \in L^\nu(m)$ for some $2 < \nu < \infty$. 
Let $0 < \alpha < 1$ and let $p = \nu/(\nu+\alpha - 1)$. Then
\begin{enumerate}  
  \item the Fourier transform $f \mapsto {\hat f}$ is bounded $L^p(m) \to H^\infty(\Sigma_{\alpha \omega_0})$;
  \item the convolution operator $\Lambda_f:g\mapsto f\ast g$ gives a bounded linear operator on $L^2(m)$ for all $f\in L^p(m)$;
  \item the map $f \mapsto T_{\alpha A}(f)$ defined via (\ref{Eq3.7}) is bounded $L^p(m) \to \mathcal{L}(E)$.
\end{enumerate}
\end{theorem}

\begin{proof} (i) The idea is that integrability of a suitable power of the positive character in $\phi_0\in {\bf S}$ enables us to extend the Fourier transform, while the Laplace representation enables us to continue the characters to analytic functions on a strip containing ${\bf S}$.
By Jensen's inequality, $\cosh (\alpha t\omega_0)\leq \cosh^\alpha (t\omega_0)$, so by H\"older's inequality we have, for $\lambda\in \Sigma_{\alpha\omega_0}$, 
  \begin{align}\label{Eq3.9} 
  \bigl\vert \phi_{\lambda}(x)\bigr\vert
      &\leq \int_{-x}^x \cosh (\alpha t\omega_0) \tau_x(dt)  \notag \\
      &\leq \Bigl( \int_{-x}^x\cosh (t\omega_0)\tau_x(dt)\Bigr)^\alpha
           \Bigl( \int_{-x}^x \tau_x(dt)\Bigr)^{1-\alpha } \notag\\
      &\leq M_0^\alpha \phi_0(x)^{1-\alpha }.
 \end{align}
By H\"older's inequality with $1/p+1/q=1$ we have $q=\nu /(1-\alpha )$. Thus
 \begin{multline}\label{Eq3.10}  
   \int_0^\infty  \vert f(x)\vert \phi_0(x)^{1-\alpha }\,m(dx)  \\
   \leq \Bigl( \int_0^\infty \vert f(x)\vert^p\, m(dx)\Bigr)^{1/p}\Bigl( \int_0^\infty \phi_0(x)^{q(1-\alpha )}\,m(dx)\Bigr)^{1/q};
  \end{multline}
where $(1-\alpha ) q=\nu >2$, and so the latest integral converges. Hence 
    \[\hat f(\lambda)=\int_0^\infty f(x) \phi_{\lambda}(x) \,m(dx)\]
converges absolutely and defines a bounded function on $\Sigma_{\alpha\omega_0}$ for all $f\in L^p(m)$. By Morera's theorem, $\hat f(\lambda )$  determines a function in $H^\infty (\Sigma_{\alpha \omega_0})$ for all $f\in L^p(m)$.

 (ii) We can in particular, apply Proposition \ref{Prop3.5} to $A:\hat g(\lambda)\mapsto \lambda\hat g(\lambda)$ and $g\in E=L^2(m)$, in which case $T_{A}(f)$ becomes the convolution operator $\Lambda_f$ by the Levitan--Plancherel theorem. By \cite[Theorem ~2.2.4]{Bloom:1994}, $\Lambda_f$ gives a bounded linear operator on $L^2(m)$, and 
   \[\Vert \Lambda_f\Vert_{{\cal L}(L^2)}=\sup \{ \vert \hat f(\phi )\vert :\phi \in {\bf S}\}\qquad (f\in L^1(m)).\]
 By (i), $\phi\mapsto \hat f(\phi )$ is bounded on ${\bf S}$ for all 
 $f\in L^p(m)\cap L^1(m)$, so we can extend to obtain $\Lambda_f\in {\cal L}(L^2)$ for all $f\in L^p(m)$.
 
(iii) By (\ref{Eq3.4}) and (\ref{Eq3.6}), we have
   \begin{align*}
   \bigl\Vert \phi_{\alpha A}(x)\bigr\Vert_{{\cal L}(E)}
     &\leq \int_{-x}^x \kappa \cosh (\alpha t\omega_0) \tau_x(dt)   \notag \\
     &\leq \kappa M_0^\alpha \phi_0(x)^{1-\alpha }
   \end{align*}
as in (\ref{Eq3.9}), so we can use (\ref{Eq3.10}) to show that $T_{\alpha A}(f)=\int_0^\infty \phi_{\alpha A}(x)f(x)m(dx)$ converges absolutely and defines a bounded linear operator for all $f\in L^p(m)$. 
\end{proof}

We now turn to the double coset hypergroup
${\bf X}={\mathrm{SL}}(2, {\bf C}))//{\mathrm{SU}}(2, {\bf C})$ mentioned in the introduction.
By \cite[p.~50]{Coifman:1976} this ${\bf X}$ has invariant measure $m(dx) = \sinh^2 x \, dx$.

\begin{corollary}\label{Cor3.8} 
Suppose that $\Vert \cos (tA)\Vert_{{\cal L}(E)}\leq \kappa \cosh t$ for all $t\in {\bf R}$. Then for $0 < \alpha < 1$ and all $f\in L^p(\sinh^2 x)$ with $1<p<2/(1+\alpha)$, 
  \[ T_{\alpha A}(f)=\int_0^\infty \frac{\sin (\alpha xA)}{\alpha A} f(x)\sinh x\, dx\]
defines a bounded linear operator on $E$.
\end{corollary}

\begin{proof} By results  of Trim\`eche (see \cite{Trimeche:1981} or \cite[p.~211]{Bloom:1994}),
there exists a commutative hypergroup on $[0,\infty )$ that has invariant measure  $2^2\sinh^2 x \, dx$. We introduce
   \begin{equation*}\label{Eq3.12} 
     \varphi_\lambda (x)
        =\frac{\sin \lambda x}{\lambda \sinh x}
        =\int_{-x}^x \frac{\cos \lambda t}{2\sinh x} \, dt
            \qquad (\lambda\in {\bf C})
   \end{equation*}
so that $\varphi_\lambda $ is a bounded
multiplicative
function for $\lambda \in \Sigma_1$ and so that
$\varphi_{\pm i}$ is
the trivial character, so that $\omega_0=1$. The Plancherel measure is
 \begin{equation*}%   This appears to be an excess label! \label{Eq3.10} 
  \pi_0(d\lambda )
    ={\frac{\lambda^2}{4\pi}}{\bf I}_{(0,\infty)}(\lambda )\, d\lambda ,  
 \end{equation*}
so that $\varphi_0(x)=x/\sinh x$ is the unique positive character in the support of $\pi_0$. Condition (\ref{Eq3.6}) holds by hypothesis, while (\ref{Eq3.4}) is immediate. Also
\[ \int_0^\infty \varphi_0(x)^\nu \sinh^2x \,dx=\int_0^\infty x^\nu \sinh^{2-\nu} x\, dx\]
converges for all $\nu >2.$ So we can apply Theorem \ref{Thm3.7} with $p = \nu/(\nu+\alpha - 1)$.  
\end{proof}

\section{Operational calculus for Sturm--Liouville hypergroups}\label{Sec:SLH}

In this section we focus on applying the operational calculus described in Section~\ref{Sec3} to hypergroups associated to certain differential operators of the form
  \[ L \phi(x) = -\frac{d^2\phi}{dx^2}
    -\frac{m'(x)}{m(x)}\frac{d\phi}{dx}, \qquad (x \ge 0).
  \]
Under suitable conditions on the function $m$, one can define a hypergroup structure on ${\bf X} = [0,\infty)$ for which the characters correspond to suitably normalized eigenfunctions of this operator. The Haar measure for these hypergroups is just $m(x)\, dx$ where $dx$ is the usual Lebesgue measure on ${\bf X}$.  

Canonical examples here include taking $m(x) = \sinh^k x$ (giving a Jacobi hypergroup as in Corollary~\ref{Cor3.8}) and Example \ref{Ex6.1}; indeed, the results are mainly of significance when 
$m(x)$ grows exponentially as $x\rightarrow\infty$. 
For our purposes, the main requirement on the hypergroup is that the characters on ${\bf X}$ have a Laplace representation. Given this, we can make use of the Fourier transform (\ref{Eq3.2}) which is entirely determined
 by $m$ and the eigenfunctions of $L$.  
 Chebli \cite{Chebli:1974} \cite{Chebli:1979} and Trim\`eche \cite{Trimeche:1981} gave sufficient condition on $m$ to ensure existence of a hypergroup structure, 
 and they also gave sufficient conditions for the characters to have a Laplace representation. See also \cite[Theorem 3.5.58]{Bloom:1994}.

\begin{definition}\label{Def4.1}  Suppose that $\omega_0\geq 0$ and $\gamma >-1/2$. We say that a function $m: [0,\infty) \to [0,\infty)$ satisfies ($H(\omega_0)$) if:
\begin{enumerate}
 \item $m(x)=x^{2\gamma+1}q(x)$ where $q\in C^\infty ({\bf R})$ is even,
positive and $m(x)/x^{2\gamma +1}\rightarrow q(0)>0$ as $x\rightarrow 0+$;
 \item $m(x)$ increases to infinity as
$x\rightarrow\infty$, and $m'(x)/m(x)\rightarrow 2\omega_0$
as
$x\rightarrow\infty$; and either
\item $m'(x)/m(x)$ is decreasing; or
  \item  the function
  \begin{equation*}\label{Eq4.1}
  Q(x)=\frac{1}{2} \Bigl(\frac{q'}{q} \Bigr)'
        +\frac{1}{4}\Bigl(\frac{q'}{q}\Bigr)^2
         +\frac{2\gamma+1}{2x} \Bigl( \frac{q'}{q}\Bigr)
          -\omega_0^2. 
  \end{equation*}
  is positive, decreasing and integrable with respect to Lebesgue measure over $(0, \infty )$.
\end{enumerate}
\end{definition}

\begin{lemma}\label{Lem4.2}
Suppose that $\omega_0 > 0$ and that $m$ satisfies $(H(\omega_0))$.
Then
\begin{enumerate}
   \item there exists a hypergroup on $[0, \infty )$ such that $x^{-} = x$;
   \item the solutions of
  \begin{equation}\label{Eq4.2}
    -\frac{d^2\phi_\lambda}{dx^2}
    -\frac{m'(x)}{m(x)}\frac{d\phi_\lambda}{dx}
    = (\omega_0^2+\lambda^2) \phi_\lambda
    \end{equation}
   such that $\phi_\lambda (0)=1$, and $\phi_\lambda'(0)=0$ for
   $\lambda \geq 0$ are characters in ${\bf S}$;
  \item $\phi_\lambda(x)$ has a Laplace representation as in (\ref{Eq3.3}),
      where $\pm i\omega_0$ corresponds to the trivial character, and the bound (\ref{Eq3.4}) holds;
  \item $\hat {\bf X}={\bf R}\cup [-i\omega_0, i\omega_0]$.
\end{enumerate}
\end{lemma}

\begin{proof}
(i) The case (iii) of Definition \ref{Def4.1} is covered in \cite{Chebli:1974}, so we emphasize case (iv). The function $\beta =q'/q$ satisfies
   \[ {\frac{1}{2}}\beta'-{\frac{1}{4}}\beta^2  +{\frac {m'\beta}{2m}}
          =\Bigl({\frac{q'}{q}}\Bigr)'+{\frac{1}{4}}\Bigl( {\frac{q'}{q}}\Bigr)^2 +{\frac{2\gamma +1}{2x}}
          =Q(x)+\omega_0^2,\]
so that $q$ satisfies SL1.1 and SL2 of \cite[p 202]{Bloom:1994}, so $m$ defines a Sturm--Liouville function the sense of \cite[Theorem 3.5.45]{Bloom:1994}.
There exists a hypergroup with convolution operation given by \cite[Section~3.5.21]{Bloom:1994}, as follows. 
The solution $u(x,y)$ of the differential equation
  \begin{equation*}\label{Eq4.3}
     - \frac{\partial^2u}{\partial x^2}  - \frac{m'(x)}{m(x)} \frac{\partial u}{\partial x} 
       = - \frac{\partial^2u}{\partial y^2}- \frac{m'(y)}{m(y)} \frac{\partial u}{\partial y}
  \end{equation*}
with initial conditions
  \[ u(x,0) = u(0,x) = f(x)
   \qquad \text{and} \qquad
   \frac{\partial u}{\partial x}(0,y) = \frac{\partial u}{\partial y}(x,0) = 0
   \]
gives $u(x,y) = \int_{\bf X} f(t)\, (\varepsilon_x \ast \varepsilon_y)(dt)$ (see \cite[2.5.35]{Bloom:1994}). Since $0 \in \mathop{\mathrm{supp}}(\varepsilon_x \ast \varepsilon_y)$, we can deduce that $x^{-} = x$ (see \cite[(HG7) p. 9]{Bloom:1994} and \cite{Zeuner:1989}). Moreover, the spectral analysis in \cite{Chebli:1974}, \cite{Chebli:1979} and \cite{Trimeche:1981} shows that ${\bf S}=[0, \infty)$.

(iii) Chebli \cite{Chebli:1974} and Bloom and Heyer \cite[Theorem 3.5.38]{Bloom:1994} showed that these eigenfunctions have a Laplace representation as in (\ref{Eq3.3}). Specifically, the
function $\lambda\mapsto \phi_\lambda (x)$ is entire,
and there exists a family of positive measures
such that
$\phi_\lambda (x)=\int_{-x}^x \cos (\lambda t)\tau_x(dt)$; in particular, $\lambda=\pm i\omega_0$ gives the trivial character and so 
(\ref{Eq3.4}) holds with $M_0=1$. 

(iv) Using Langer's transformation \cite[p. ~5]{Chebli:1979}, we let
$\phi_\lambda (x)=\psi_\lambda(x)/\sqrt{m(x)}$.
Then $\psi_\lambda$ satisfies
 \begin{equation}\label{Eq4.4}
   -\psi''_\lambda(x)+\Bigl(\frac{m''}{2m}
      -\Bigl(\frac{m'}{2m} \Bigr)^2
      -\omega_0^2\Bigr) \psi_\lambda (x)
    = \lambda^2\psi_\lambda(x),
  \end{equation}
that is   
\begin{equation}\label{Eq4.5}
   -\psi''_\lambda(x)+\Bigl(\frac{4\gamma^2-1}{4x^2} +Q(x)\Bigr) \psi_\lambda (x)
    = \lambda^2\psi_\lambda(x).
  \end{equation}  
Hence $\phi_\lambda (x)$ is real, if and only if $\lambda^2\in {\bf R}$; that is $\lambda\in {\bf R}\cup i{\bf R}$. By comparing (\ref{Eq4.5}) with the sine equation as in \cite[Theorem 1.5.7]{Hille:1969}, we see that $\phi_0(x)\rightarrow 0$ as $x\rightarrow\infty$.

For all $\nu>1$, we have by two application of H\"older's inequality
   \begin{align*}
     1 &= \int_{-x}^x\cosh (t\omega_0)\, \tau_x(dt)  \\ 
       &\leq \Bigl( \int_{-x}^x \cosh^\nu (\omega_0t)\,\tau_x(dt)\Bigr)^{1/\nu}
                                \Bigl( \int_{-x}^x \tau_x(dt)\Bigr)^{(\nu -1)/\nu} \\
       &\leq \Bigl( \int_{-x}^x \cosh (\nu\omega_0t)\,\tau_x(dt)\Bigr)^{1/\nu}
                   \Bigl( \int_{-x}^x \tau_x(dt)\Bigr)^{(\nu -1)/\nu},
    \end{align*}
which implies that $\phi_{i\nu \omega_0}(x)\geq \phi_0(x)^{1-\nu }$.  Hence $\phi_{i\nu\omega_0}(x)\rightarrow\infty$ as $x\rightarrow\infty$, so $\phi_{i\nu \omega_0}$ does not belong to $\hat {\bf X}$. Hence $\hat {\bf X}={\bf R}\cup [-i\omega_0, i\omega_0].$  
\end{proof}

Our aim is to now define ${\hat f}(A)$ for suitable $A$ and $f$ via the Fourier transform for such a Sturm--Liouville hypergroup. 

\begin{theorem}\label{Thm4.3} Suppose that $m$ and $\phi_\lambda$ are as in Lemma~\ref{Lem4.2} with $\omega_0>0$ and that $(\cos (tA))_{t\in {\bf R}}$ is a strongly continuous cosine family on
$E$ such that
   \begin{equation}\label{Eq4.7}
      \Vert \cos (tA)\Vert_{{\cal L}(E)}\leq \kappa \cosh (\omega_0t)\qquad (t\in {\bf R})
   \end{equation}
and some $\kappa<\infty$. Let $2<\nu <\infty$, $0<\alpha <1$ and $p=\nu /(\nu +\alpha -1)$.
Then 
\begin{enumerate}
    \item there exists a commutative hypergroup $({{\bf X}}, \ast )$ on $[0, \infty )$ such that  $\phi_\lambda$ is a
bounded multiplicative function on $({{\bf X}}, \ast )$ for all $\lambda\in \Sigma_{\omega_0}$;
    \item  the Fourier transform
$f\mapsto \hat f(\lambda )$ is bounded $L^p(m)\rightarrow H^\infty (\Sigma_{\alpha \omega_0})$;
    \item $(\phi_A(x))_{x\geq 0}$ as in (\ref{Eq3.5}) 
    gives a bounded
    family of linear operators on $E$, $T_A (f)=\int_0^\infty f(x)\phi_A(x)m(x)\,dx$ %as in (\ref{Eq3.7})
    defines a bounded linear operator on $E$ for all $f\in L^1(m)$, and $T_{A}(f\ast g)=T_{A} (f)T_{A}(g)$ for all $f, g\in L^1(m)$;
    \item the map $f \mapsto T_{\alpha A}(f)$ defined via (\ref{Eq3.7}) is bounded $L^p(m) \to \mathcal{L}(E)$.
\end{enumerate}
\end{theorem}

\begin{proof}
(i) This follows from Lemma \ref{Lem4.2}.

(ii) By comparing (\ref{Eq4.5}) with the sine equation, as in \cite[p. 527]{Hille:1969} one obtains a bound
$\psi_\lambda (x)=O(e^{\eta x})$ as $x\rightarrow\infty$
where $\eta =\vert\Im \lambda\vert >0$. In particular, we have
    \begin{equation*}
      \int_0^\infty \vert \phi_0(x)\vert^\nu m(x)\,dx=\int_0^\infty \bigl\vert \psi_0 (x)\bigr\vert^{\nu}m(x)^{1-(\nu/2)}\, dx
    \end{equation*} 
which converges for $2<\nu <\infty$. By Lemma~\ref{Lem4.2}, $\phi_\lambda$ is a bounded
multiplicative function for
$\lambda \in \Sigma_{\omega_0}$, and has a Laplace representation. Hence we can apply Theorem \ref{Thm3.7}(ii). Note that for $\lambda >0$, all solutions of (\ref{Eq4.5}) oscillate boundedly, so $\phi_\lambda$ is not in $L^2(m)$. Thus we cannot extend this proof to the case $\nu =2$.

(iii) By Lemma \ref{Lem4.2}, the hypergroup has a Laplace representation. Condition (\ref{Eq3.4}) holds since the trivial character arises for $\lambda =i\omega_0$ so the Laplace representation gives $\int_{-x}^x\cosh t\omega_0\tau_x(dt)=1$,
while (\ref{Eq3.6}) holds by hypothesis. Thus all the hypotheses
of Proposition~\ref{Prop3.5} apply.

(iv) Theorem \ref{Thm3.7}(iii) applies.
\end{proof}

Trim\`eche \cite[section 8]{Trimeche:1981} considers the difference operators 
    \[  \sigma_t f(x) =  \tfrac{1}{2}(f(x+t)+f(x-t))\qquad (x,t\in {\bf R})\]
 in relation to the Fourier transform for certain Jacobi hypergroups. Definition \ref{Def4.1} does not cover the Jacobi hypergroups with $m(x)=\cosh^k x$, since $\gamma =-1/2$ is excluded. However, such 
examples are otherwise addressed by the following result, which enables one to use the transference theorem for locally bounded groups from \cite{Blower:1996}. To clarify the various operations, we introduce
   \begin{equation*}
     {\cal X}f(x)=\int_{-x}^x f(t)\tau_x(dt)
   \end{equation*}
for $f \in C_{c,ev}^\infty ({\bf R}; {\bf R}))$, the compactly supported and even functions in $C^\infty ({\bf R}; {\bf R})$. 
For $t \in {\bf R}$, let $S_t$ denote the translation operator $S_tf(x) =  f(x-t)$.

\begin{proposition}\label{Prop4.4} 
Suppose that $q\in C^\infty ({\bf R})$ is positive and even, and that there exist $\kappa_1, \kappa_2$ such that $\kappa_1\leq q'(x)/q(x)\leq \kappa_2$ for all $x\in {\bf R}$. Let $1\leq p<\infty$. Then
\begin{enumerate}
  \item
   $\displaystyle
      \cos (t\sqrt{L}){\cal X}f = \int_{-x}^x \sigma_tf(s) \, \tau_x(ds),
                 \quad (f\in C_{c,ev}^\infty ({\bf R}; {\bf R}))$;
  \item $(S_t)_{t\in {\bf R}}$ defines a $C_0$ operator group on 
$L^p({\bf R}; q(x)\, dx)$ such that 
    $\Vert S_t\Vert_{{\cal L}(L^p)}\leq M_pe^{w_p\vert t\vert }$ for all $t \in {\bf R}$,
   where $w_p=\max \{ \vert \kappa_1\vert, \vert \kappa_2\vert\}/p$;
  \item there exists a generator $A$ such that $\cos (tA) = \tfrac{1}{2}( S_t+S_{-t})$ for $t\geq 0$ defines a strongly continuous cosine family on $L^p({\bf R}; q(x)\, dx)$ satisfying (\ref{Eq3.6}).
\end{enumerate}
\end{proposition}

\begin{proof} (i) Trim\`eche \cite{Trimeche:1981} has a similar result in different notation, so we give the proof for completeness. Observe that ${\cal X}: \cos s\lambda\mapsto \phi_\lambda (x)$ by the Laplace representation (\ref{Eq3.3}), and $\cos (t\sqrt{L})\phi_\lambda (x)=\cos (t\lambda )\phi_\lambda (x)$ for $\phi_\lambda\in {\bf S}$ by the spectral theorem. Now $\sigma_t: \cos (s\lambda)\mapsto \cos (t\lambda) \cos (s\lambda )$. Hence the required identity holds for $\cos (s\lambda )$, and then we can use the Fourier cosine transform to obtain the stated result. 

(ii) We have
\begin{equation}\int_{-\infty}^\infty \vert S_tf(x)\vert^pq(x)\, dx=\int_{-\infty}^\infty \vert f(x)\vert^pq(x+t)\, dx\end{equation}
so it suffices to bound $q(x+t)/q(x)$ from above for all $x$ in terms of $t$. This splits into cases according to the signs of $x$ and $t$ which are all elementary estimates.

(iii) This follows from (ii) by \cite[Remark 8.11]{Goldstein:1985}.
\end{proof}     

\begin{remark}\label{Rem4.5} Consider the case of Definition \ref{Def4.1} in which $q=1$, so that $m(x)=x^{2\gamma +1}$ and $\omega_0=0$. Then $L^p(m)$ has a strongly continuous group $(V_t)_{t\in {\bf R}}$ of dilation operators $V_t:f(x)\mapsto e^{(2\gamma +2)t/p}f(e^tx)$ for $1\leq p<\infty$, such that $\Vert V_tf\Vert_{L^p}=\Vert f\Vert_{L^p}$ for all $t\in {\bf R}$ and $f\in L^p(m)$. The transference theory of \cite{Berkson:1987} applies to this dilation group.

Let $J_\gamma$ denote Bessel's function of the first
kind of order $\gamma$ and define
  \[ \psi_\lambda (x)=\lambda^{-\gamma}x^{1/2}2^\gamma J_\gamma
        (\lambda x)
      =\frac{\Gamma (\gamma +1)x^{\gamma +1/2}}{\Gamma (1/2)
                                  \Gamma (\gamma+1/2)  }
       \int_{-x}^x \Bigl(1-\frac{s^2}{x^2}\Bigr)^\gamma
               \frac{\cos s\lambda }{\sqrt{x^2-s^2}}\, ds, \]
so that $\lambda\in {\bf R}$,
$\lambda \mapsto \psi_\lambda (x)$ is entire
and of exponential type, and
  \[ -\psi_\lambda''(x)+\frac{4\gamma^2-1}{4x^2}\psi_\lambda(x)
    =\lambda^2\psi_\lambda(x). \]
 The hypergroup associated with $J_0$ is studied by detail by Jewett \cite{Jewett:1975}, who finds that the trivial character lies in ${\bf S}$. 
Taylor uses the operational calculus associated with Bessel functions of the first kind \cite[p.~1120]{Taylor:2009} to obtain bounds on certain differential operators associated with the wave equation on Euclidean space. Fractional integration operators for the Hankel--Bessel transform are discussed in \cite[section 5]{Trimeche:1981}. By contrast, the examples in the following sections have $\omega_0>0$.
\end{remark}

\section{Fractional integration of cosine families}\label{Sec:MF}

Several more classical transforms and associated families of functions fall within this framework. In this section we look at the case where $m(x) = \sinh x$. The hypergroup Fourier transform in this setting is the Mehler--Fock transform of order zero.

\begin{definition}\label{Def5.1}
\begin{enumerate}
  \item For $m,n=0,1, \dots,$ the \textit{associated Legendre functions} may be defined as in \cite[p.156]{Erdelyi:1953v1} to be the functions $P_{\nu}^\mu$ such that
\[ P_\nu^\mu (\cosh x )=\sqrt{ {\frac{2}{\pi}}}{\frac {(\sinh x)^\mu}{\Gamma ((1/2)-\mu)}}\int_0^x {\frac{\cosh (\nu+(1/2))y}{(\cosh x-\cosh y)^{\mu +(1/2)}}} dy.\] 
  \item \textit{Legendre's functions} are defined by
 \[ \phi_\lambda (x)= P_{i\lambda  -(1/2)}(\cosh x)
   = \frac{1}{\pi \sqrt{2}}
   \int_{-x}^x \frac{\cos \lambda y}{\sqrt{\cosh x-\cosh y}}\, dy\qquad
(\lambda\in {\bf C} ). \]
See \cite[(7.4.1)]{Sneddon:1972}. An alternative notation is $R^{(0,0)}_z=P_z$ with $z=i\lambda -(1/2)$ as in \cite[p.~68]{Trimeche:1981}.  
  \item
The \textit{Mehler--Fock transform of order zero} of $f\in L^1(\sinh x\, dx)$ is
 \[ \hat f(\lambda )=\int_0^\infty f(x)\phi_\lambda (x)\sinh x\, dx.\]
\end{enumerate}
\end{definition}

Legendre's functions are associated with Laplace's equation in toroidal
coordinates, and sometimes called toroidal functions; see \cite{Mehler:1881,Sneddon:1972}. Further details of the Mehler--Fock transform of order zero can be found in \cite[p.~390]{Sneddon:1972}.

\begin{proposition}\label{Prop5.1}
Let $(\cos (tA))_{t\in {\bf R}}$ be a cosine family on $E$ and suppose that there exists $\kappa$  such that $\Vert \cos (tA)\Vert_{{\cal L}(E)}\leq \kappa \cosh (t/2)$
 for all $t\geq 0$. Then
 \begin{enumerate}
 \item there exists a hypergroup
 $([0, \infty ),\ast )$ with Laplace representation (\ref{Eq3.3}) such that $f\mapsto \hat f$ is the Mehler--Fock transform of order zero;
 \item  $(\phi_A(x))_{x>0}$ is a bounded family of operators; 
 \item  the integral 
    \begin{equation}\label{Eq5.1} 
      T_{A}(f)=\int_0^\infty \phi_{A}(x)f(x)\sinh x\, dx\qquad (f\in L^1(\sinh x \, dx)) 
    \end{equation}
      defines a bounded linear operator such that $T_A(g\ast h)=T_A(g)T_A(h)$ for all $g,h\in L^1(\sinh x\, dx)$; 
  \item for $2< \nu < \infty$, $0 < \alpha < 1$ and $p = \nu/(\nu+\alpha - 1)$, the linear operator 
$f\mapsto T_{\alpha A}(f)$ is bounded $L^p(\sinh x\, dx) \to \mathcal{L}(E)$.
\end{enumerate}
\end{proposition}

\begin{proof} (i) Mehler \cite[(8b) of page 184]{Mehler:1881} showed
that
  \[ -\phi_\lambda''(x) - \coth x\, \phi_\lambda'(x)
    = (\lambda^2+({1/4}))\phi_\lambda(x).\]
Trim\`eche \cite{Trimeche:1981} introduces a
hypergroup structure on $(0, \infty )$ such that the
$\phi_\lambda $ for
$\lambda\in \Sigma_{1/2}$ are bounded and multiplicative
for this hypergroup, and he shows that the invariant measure
and the Plancherel measure are
supported on $[0, \infty)$, and satisfy
  \begin{equation}\label{Eq5.2}
  m(x)\,dx = \sinh x\,dx,\quad
  \pi_0 (d\lambda)=
  \frac{2\vert \Gamma ((1/4)+(i\lambda /2))
    \Gamma((3/4)+(i\lambda /2))\vert^2}{\vert
    \Gamma(i\lambda /2)
    \Gamma(1+(i\lambda /2))\vert^2} \, d\lambda.
    \end{equation}
By a computation involving $\Gamma$ functions,
particularly the identity
$-z\Gamma (-z)\Gamma (z)=\pi \mathop{\mathrm{cosec}} (\pi z)$,
one can
reduce (\ref{Eq5.2}) to $\pi_0(d\lambda ) =\lambda \tanh (\pi \lambda)
d\lambda$, so the generalized Fourier transform $\hat f(\lambda
)=\int_0^\infty f(x)\phi_\lambda (x)m(x)\,dx$ reduces to the
Mehler--Fock transform of order zero. Note that $\lambda =i/2$ gives the
trivial character, which is not in the support of $\pi_0$. 

(ii) Definition \ref{Def5.1} gives the Laplace representation. We now observe that
  \[ \int_{-x}^x \frac{\cosh (y/2)\,dy}{\sqrt{\cosh x-\cosh y}}
      =\int_{-x}^x \frac{\cosh (y/2)\, dy}{\sqrt{\sinh^2(x/2)-\sinh^2(y/2)}}
  \]
is bounded, so (\ref{Eq3.4}) holds, while (\ref{Eq3.6}) holds by hypothesis. Hence Proposition \ref{Prop3.5} gives $\Vert \phi_A(x)\Vert_{{\cal L}(E)}\leq \kappa $.

(iii) Given that the hypergroup convolution
$\ast$ exists, we can apply Proposition \ref{Prop3.5}.

(iv) Whereas $\phi_0(x)$ can be expressed in terms of Jacobi's complete
elliptic integral of the first kind with modulus $i\sinh (x/2)$, we require only the formula
   \begin{equation*}\label{Eq5.3}
      \phi_0(x)={\frac{1}{\pi}}\int_0^x{\frac{dy}{\sqrt{\sinh^2(x/2)-\sinh^2(y/2)}}} 
               \leq {\frac{2\sqrt{2x}}{\pi\sqrt{\sinh (x/2)}}}.
   \end{equation*}
From the differential equation (\ref{Eq4.5}), we obtain $\phi_0(x)=O(xe^{-x/2})$ as $x\rightarrow\infty$, so $\phi_0\in L^\nu (\sinh x)$ for all $2<\nu <\infty$. Hence we can apply Theorem \ref{Thm3.7}. 

\end{proof}

\begin{example}\label{Ex5.3} One can compute the transforms of polynomials in ${\mathrm{sech}}\, (x/2)$ by contour integration. For example, one can adapt the formulae in \cite{Sneddon:1972} to obtain the array of Mehler--Fock transforms
  \[\begin{array}{rl}
f(x)&\qquad  \hat f(\lambda )\cr
{\mathrm{sech}}\, (x/2)&\qquad (2/\lambda ){\mathrm{cosech}}\,(\pi\lambda )\cr
({\mathrm{sech}}\, (x/2))^{3} &\qquad  8\lambda {\mathrm{cosech}} (\pi \lambda )\cr
({\mathrm{sech}}\,(x/2))^5&\qquad (16/3)\lambda^3{\mathrm{cosech}}\,(\pi\lambda )
 \end{array}
 \]
in which the last two transforms are bounded and holomorphic on $V_{\phi, 1}$ for all $0<\phi<\pi/2$. Likewise, any positive even power $({\mathrm{sech}}\, (x/2))^\nu $ transforms to a constant multiple of $\lambda^{\nu-2} {\mathrm{sech}}(\pi \lambda )$.
\end{example}

In the Cauchy problem for the Euclidean wave equation in space dimension
three, the solution can have one order of differentiability fewer than
the initial data, due to the possible formation of caustics. Hence it
is natural to apply fractional integration operators to the cosine
families which address this
possible loss of smoothness, and the order of the
fractional integration required
can depend directly upon the dimension. The operators that we require are described
in the following lemma.

\begin{definition}\label{Def5.4} The fractional integration operators
$W_\alpha$ and $U_\beta$ are
defined on $C^\infty ({\bf R})$ by
 \begin{align*}
    W_\alpha f(x) &= \frac{1}{\Gamma (\alpha )}
      \int_0^x (\cosh x-\cosh t)^{\alpha -1} \sinh t \, f(t)\, dt, \\
    U_\beta f(x) &= \frac{1}{\Gamma (\beta )}
       \int_0^x (\cosh x-\cosh t)^{\beta -1} f(t)\, dt,
\end{align*}
where $\alpha$ and $\beta$ are the orders of $W_\alpha$ and $U_\beta$,
such that $\Re \alpha >0$ and $\Re \beta >0$.
\end{definition}

\begin{lemma}\label{Lem5.5}
\begin{enumerate}
\item
Let $Df=f'$. Then the operators satisfy
  \[ W_\alpha W_\beta =W_{\alpha +\beta}, \quad
  W_\alpha U_\beta =U_{\alpha +\beta}, \quad
  \mathop{\mathrm{cosech}} x \,DW_1=I, \quad DU_1=I. \]
\item For $\nu\in {\bf Z}$ such that $\nu\geq 0$ and $\lambda\in {\bf R}$, the associated
 Legendre function satisfies
  \begin{equation}\label{Eq5.4}
     U_{\nu+1/2}(\cos (x\lambda ) )=\sqrt\frac{\pi}{2}
     \frac{\Gamma (1/2+i\lambda -\nu)}{\Gamma (1/2+i\lambda +\nu)} (\sinh x)^\nu 
                  P_{i\lambda -1/2}^\nu (\cosh x),
  \end{equation}
where the quotient of Gamma functions is a rational function of $\lambda$, and
   \begin{equation}\label{Eq5.5} 
   W_{\nu-1/2}(\cos (x\lambda ))
    =\frac{d}{dx} U_{\nu+1/2} (\cos (x\lambda))\qquad (\nu\in {\bf N} ).
   \end{equation}
\end{enumerate}
\end{lemma}

\begin{proof}
(i) This is essentially contained in the statement and proof of \cite[Lemma 5.2]{Lax:1982}. 
See also \cite[Theorem 5.2]{Trimeche:1981}.

(ii) The identity (\ref{Eq5.4}) is known as the Mehler--Dirichlet formula \cite[p.~373, 381]{Sneddon:1972}, from which we obtain (\ref{Eq5.5}) by differentiating.
\end{proof}

For these operator families, we have the following result.

\begin{proposition}\label{Prop5.6} Suppose that $(\cos (tA))_{t\in {\bf R}}$ is strongly continuous cosine family on a Banach space $E$.
 Suppose that there exists $M>0$ such that $\Vert \cos (tA)\Vert_{{\cal L}(E)}\leq M\cosh (t/2)$ for all $t\in
{\bf R}$. Then $(U_{1/2}(\cos (tA))_{t\in {\bf R}}$ is a bounded family of operators.
\end{proposition}

\begin{proof} By Lemma~\ref{Lem5.5}, the trigonometric and Legendre functions of Definition~\ref{Def5.1} are
related by
  \[ \phi_\lambda (x)=\sqrt{\frac{2}{\pi}} U_{1/2}
  (\cos \lambda x),\quad
  \cos\lambda x=\frac{d}{dx} \sqrt{\frac{\pi}{2}}
  W_{1/2}(\phi_\lambda(x)).\]
In the notation of Proposition~\ref{Prop5.1},
 we have
$\phi_A(t)=U_{1/2}(\cos (tA))$, whence the result.
\end{proof}

\section{Geometrical applications}\label{Sec:Appl}

In this final section we shall look at certain Laplacian operators which occur naturally in differential geometry and show how the results of the earlier sections can be applied in these settings. For the wave equation associated with the Laplacian operator on a Riemannian manifold, the fundamental solutions travel at unit speed.
We can therefore accommodate the growth of balls by incorporating a suitable weight $m(x)$ in the functional calculus.

\begin{example}\label{Ex6.1} (i) As a model for hyperbolic space ${\cal{H}}^n$ of dimension $n\geq 2$, we use the upper half-space
  \[ {\cal{H}}^n=\{ x=(\xi,t): \xi\in {\bf R}^{n-1}, t>0\} \]
with metric $dx^2=t^{-2}(d\xi^2+dt^2)$ and volume measure ${\mathrm{vol}}_{\mathcal{H}} (dx)=t^{-n}dtd\xi$. 
Let $S(x,r)$ be the hyperbolic sphere of radius $r$ and centre $x$. 
  The Laplacian in geodesic polars at $x$ is 
\[\Delta =-\frac{\partial^2}{\partial r^2}-(n-1) \mathop{\mathrm{coth}} r \frac{\partial}{\partial r} +\Delta_{S(x,r)}\]
 where $\Delta_{S(x,r)}$ is the Laplacian on $S(x,r)$. We restrict attention to radial functions depending on $r$. The corresponding hypergroup on $(0, \infty )$ is
  \begin{equation} 
  \varepsilon_r\ast\varepsilon_s
  =   {\frac{\Gamma (n/2)}{\sqrt{\pi} \Gamma ((n-1)/2)}}
      \int_0^\pi \varepsilon_{\cosh^{-1} (\cosh r\cosh s+\sinh r\sinh s\cos \theta)} \sin^{n-2}\theta\, d\theta.
      \end{equation}
and the invariant measure is $\sinh^{n-1}r\, dr$.

(ii) Let 
   \begin{equation}
   \sigma_\kappa (r)
       ={\frac{n\pi^{n/2}}{\Gamma (n/2+1)}}\Bigl({\frac{\sinh r\sqrt{-\kappa}}{\sqrt{-\kappa}}}\Bigr)^{n-1}
              \qquad (\kappa <0).
    \end{equation}
and $m_{\kappa}(r) = \int_0^r \sigma_\kappa (s)\,ds$. 
When $x=(\xi,1)$, $S(x,r)$ is also a Euclidean sphere of centre $(\xi,\cosh r)$ and radius $\sinh r$ and
hence has area $\sigma_{-1}(r)$; see \cite{Chavel:2001}.

(iii) The functions $\log \sigma_{-1} (x)$ and 
$\log m_{-1}(x)$ are concave on $(0, \infty )$ for all $n\in {\bf N}$. To see this for $\log m_{-1}(x)$ we write
   \[ h_0(x)=n\cosh x\int_0^x \sinh^nt\, dt-\sinh^{n+1}x\qquad (x\geq 0),\]
and compute
   \[ \frac{d^2}{dx^2}\log m_{-1}(x)= {\frac{\sinh^{n-1}x}{\bigl(\int_0^x \sinh^nt\, dt\bigr)^2}} h_0(x)\qquad (x>0),\]
so it suffices to prove that $h_0(x)\leq 0$ for all $x>0$. Since $h_0(0)=0,$ it suffices to show that $h_0'(x)\leq 0$ for $x\geq 0$. For $n=1$, this is easy to check. For $n\geq 2,$ we have
$h_0'(x)/\sinh x=h_1(x),$ where
   \[ h_1(x)=n\int_0^x \sinh^nt\, dt-\cosh x\sinh^{n-1}x\qquad (x\geq 0).\] 
Now $h_1(0)=0$ and $h_1'(x)=-(n-1)\sinh^{n-1}x\leq 0$, so $h_1(x)\leq 0$ for $x\geq 0$; hence $h_0'(x)\leq 0$, and so $h_0(x)\leq 0$, as required. This shows that hyperbolic space satisfies all the hypotheses of Proposition \ref{Prop6.3} below.
\end{example}

\begin{proposition}\label{Prop6.2}
For $2\leq \nu\leq \infty$, $\alpha =(\nu-2)(n-1)/(2\nu)$ and 
\[\max \{n\nu /((n+1)\nu +2-2n),1\}<p<\nu/(\nu-2)\] the integral 
  \begin{equation*}
       \int_0^\infty f(t) U_{\alpha +1}(\cos t\sqrt{\Delta}) \, dt
  \end{equation*}
defines a bounded linear operator on $L^\nu({\mathrm{vol}}_{\mathcal{H}})$ for all $f\in L^p(\sinh^{n-1} t\, dt)$.
\end{proposition}

\begin{proof}
For $n$ even let
  \begin{equation*} 
     w(x,t)=\frac{\pi^{-n/2}}{2^{(n+1)/2}}\Bigl( \frac{1}{\sinh t}\frac{\partial}{\partial t}\Bigr)^{(n-2)/2}
                 \int_{B(x,t)} \frac{u(y)\,{\mathrm{vol}}_{\mathcal{H}} (dy)} {\sqrt{\cosh t-\cosh \rho (x,y)}}
  \end{equation*}
where $\rho (x,y)$ denotes the hyperbolic distance between $x$ and $y$; for $n$ odd, 
let
  \begin{equation*} 
     w(x,t)=\frac{\pi^{(1-n)/2}}{2^{(n+1)/2}}\Bigl( \frac{1}{\sinh t}\frac{\partial}{\partial t}\Bigr)^{(n-3)/2} 
        \frac{1}{\sinh t} \int_{S(x,t)} u(y)\,{\mathrm{area}}_{S(x,t)}(dy)
 \end{equation*}
where ${\mathrm{area}}_{S(x,t)}$ is the area measure on $S(x,t)=\partial B(x,t)$.  Then $w$ satisfies the wave equation on hyperbolic space with
  \[ \frac{\partial^2}{\partial t^2}w(x,t)=-\Delta w(x,t)  \]
with $w(x,0)=0$ and $\frac{\partial w}{\partial t}(x,0)=u(x)$. Hence we can write 
  \begin{equation*}
    w(x,t) = U_1(\cos (t\sqrt{\Delta})) u= \frac{ \sin t\sqrt{\Delta}}{\sqrt{\Delta}} u
  \end{equation*}
so that $U_{\alpha +1}=W_{\alpha }U_{1}$ by Lemma \ref{Lem5.5}, and proceed to bound these operators. 

The family of operators 
   \begin{equation*}
   T(\alpha ;t)=\Gamma (\alpha +1)U_{\alpha +1}(\cos (t\sqrt{\Delta }))
   \end{equation*} 
is bounded and analytic on  $\{\alpha :0<\Re \alpha <(n-1)/2\}$ in the sense that 
  \begin{equation*}
     \alpha \mapsto\int_{{\mathcal{H}}^n} T(\alpha ;t)f(x)g(x) \,\mathrm{vol}_{\mathcal{H}}(dx)
  \end{equation*}
is analytic for all $t>0$ and all compactly supported smooth functions $f$ and $g$, and bounded and continuous on $\{\alpha :0\leq \Re \alpha \leq (n-1)/2\}$ for all $t>0$. Indeed, the operator
  \begin{equation*} 
      T(i\tau ;t): f\mapsto \int_0^t (\cosh t-\cosh s)^{i\tau} \cos s\sqrt{\Delta} f(x)\, ds
  \end{equation*}
is bounded on $L^2({\mathrm{vol}}_{\mathcal{H}})$ by the spectral theorem. Also, writing $\rho (x,y)=s$, we have an operator on $L^\infty$
  \begin{equation*} 
       T((n-1)/2+i\tau ;t):f\mapsto \int_{B(x,t)}(\cosh t-\cosh s)^{i\tau}f(y)\,{\mathrm{vol}}_{\mathcal{H}}(dy)
  \end{equation*}
% which is bounded on $L^\infty ({\mathrm{vol}}_{\mathcal{H}})$ by Schur's lemma. Indeed the 
with norm bounded by
$\mathrm{vol}_{\mathcal{H}}(B(x,t)) \leq M_n t \sinh^{n-1} t$ for some $M_n>0$ and all $t\geq 0$.

By Stein's interpolation theorem \cite[p. 69]{Stein:1970}, $T(\alpha ;t)$ and hence $U_{\alpha +1}(\cos t\sqrt{\Delta})$ are bounded linear operators on $L^\nu({\mathrm{vol}}_{\mathcal{H}})$ for $1/\nu=\theta /2+(1-\theta )/\infty$ and $\alpha =0\theta +(1-\theta )(n-1)/2$, with norm
   \begin{align*}
    \Vert T(\alpha ;t)\Vert_{{\cal L}(L^\nu)}
       &\leq C\sup_\tau \Vert T(i\tau ;t)\Vert_{{\cal L}(L^2)}^\theta 
                  \sup_\tau \Vert T((n-1)/2+i\tau ;t)\Vert_{{\cal L}(L^\infty)}^{1-\theta}  \\
       &\leq C t \sinh^{(n-1)(1-\theta )}t
   \end{align*}
for some $C>0$. Now take $p$ as above, and observe that for $f\in L^p(\sinh^{n-1} t\, dt)$ we have
  \begin{multline}\label{Eq6.12}
   \int_0^\infty t\sinh^{(n-1)(1-2/\nu)}t\, \vert f(t)\vert \, dt\\
      \leq \Bigl( \int_0^\infty \vert f(t)\vert^p\sinh^{n-1} t\, dt\Bigr)^{1/p}\Bigl( \int_0^\infty t^{p/(p-1)}\sinh^{-r}t\, dt\Bigr)^{(p-1)/p},
  \end{multline}
where $r=p(n-1)(1/p-1+2/\nu )/(p-1)>0$ since $1/p-1+2/\nu>0$, and $p/(p-1)-r>-1$ since 
\[{\frac{p}{p-1}}-r+1={\frac{p}{p-1}}\Bigl( {\frac{n\nu +\nu -2n+2}{\nu }} -{\frac{n}{p}}\Bigr) >0\]
so the final integral in (\ref{Eq6.12}) converges. 
When $n$ is even, $U_{(n+1)/2}(\cos (t\sqrt {\Delta }))$ is given in terms of associated Legendre functions by (\ref{Eq5.4}).
\end{proof}

The preceding example is the fundamental basis for comparison, as follows.
Let ${\cal{M}}$ be a complete Riemannian
manifold of dimension $n$ with metric $\rho$ that has 
injectivity radius bounded below by some $r_0>0$. This
means that
the exponential map is injective on the tangent space above the ball 
$B(x, r_0)=\{ y\in {\cal{M}}: \rho (x,y)\leq r_0\}$
for all $x\in {\cal{M}}$; see \cite{CheegerGT:1982}. For fixed $x_0\in {\cal{M}}$, we
can use $\rho (x,x_0)$ as the radius in a system of polar coordinates with centre $x_0$,
noting that $\rho$ is not differentiable on the cut locus. Let ${\mathrm{vol}}$ be the Riemannian volume measure, and 
for an open subset $\Omega$ with compact closure, let $\Omega_\varepsilon=\{ x\in {\cal M}:\exists y\in \Omega : \rho (x,y)\leq \varepsilon\}$ be its $\varepsilon$-enlargement for $\varepsilon > 0$. Then let the outer Hausdorff measure of the boundary $\partial\Omega$ of $\Omega$ be  
    \[  \mathrm{area}(\partial\Omega )
    = \lim \sup_{\varepsilon\rightarrow 0+} \varepsilon^{-1}(\mathrm{vol}(\Omega_\varepsilon )-\mathrm{vol}(\Omega)).
    \] 
In particular, let $\sigma (x_0,r)=\mathrm{area}(\partial B(x_0, r))$ be the surface area of a sphere, and $m(x_0,r)={\mathrm{vol}}(B(x_0,r))$ the volume of a ball.

The Laplace operator $\Delta$ is essentially self-adjoint on $C_c^\infty ({\cal{M}}; {\bf C})$
by Chernoff's theorem
\cite{Chernoff:1973}, so we can define functions of $\sqrt{\Delta}$ via the spectral theorem in
$L^2({\cal{M}},{\hbox{vol}})=L^2({\cal{M}})$. By the spectral theorem, one can define the group of imaginary powers $\Delta ^{i\tau }$ which forms a $C_0$ group on $L^2({\cal M})$. Furthermore, $\Delta ^{i\tau }$ extends to define a $C_0$ group on $L^p({\cal M})$ for $1<p<\infty$, as discussed in \cite{Strichartz:1982}, especially Theorem 4.5. Hence Proposition~\ref{Prop2.2} applies to $A^{i\tau} =\Delta ^{i\tau}$.

Then by \cite[(1.17)]{CheegerGT:1982}, for any smooth radial function $g(r)$, the Laplace operator satisfies 
   \begin{equation}
       \Delta g=-g''(r)-{\frac {\sigma'(x_0,r )} {\sigma (x_0,r )}} g'(r).
   \end{equation} 
We formulate conditions under which this differential operator on $(0, \infty )$ lies in the scope of section 4. 
Condition (i) of Definition~\ref{Def4.1} relates to local geometrical properties with small $r>0$; whereas (ii) relates to global geometry and large $r$.

For $r_0>\delta>0$, the modified Cheeger constant \cite{Chavel:2001} is
   \begin{equation}
       I_{\infty ,\delta}({\cal M})=\inf \Bigl\{ {\frac { {\mathrm{area}}(\partial \Omega)}{{\mathrm{vol}}(\Omega )}}: \Omega\Bigr\}
   \end{equation}
where the infimum is taken over all the open subsets $\Omega$ of ${\cal M}$ that have compact closure, have smooth boundary $\partial\Omega$ and contain a metric ball of radius $\delta$. 

\begin{proposition}\label{Prop6.3}
Let the Riemannian manifold ${\cal M}$ be as above and suppose that
\begin{enumerate} 
  \item ${\cal M}$ is noncompact with Ricci curvature bounded below by ${\kappa }(n-1)$ where $\kappa <0$;
  \item the modified Cheeger constant satisfies $I_{\infty ,\delta}({\cal M}) > 0$ for some $\delta > 0$;
  \item $r\mapsto \log m(x_0,r)$ and $r\mapsto \log \sigma (x_0,r)$ are concave functions of $r\in (0, \infty )$. 
\end{enumerate}
Then $m(x_0,r)$ and $\sigma (x_0,r)$ satisfy conditions (i), (ii) and (iii) of Definition~\ref{Def4.1} with $2 \omega_0 \ge I_{\infty ,\delta}({\cal M})$.
\end{proposition}  

\begin{proof} First consider small $r>0$. By Bishop's comparison theorem \cite[p.~126]{Chavel:2001} and the 
local isoperimetric inequality with constant $S_D>0$ as in \cite[p.~130]{Chavel:2001}, there exists $r_0>0$ such that
  \begin{equation*}
       S_Dm(x_0,r)^{(n-1)/n}\leq \sigma (x_0,r)\leq \sigma_\kappa (r)\qquad (0<r<r_0).
  \end{equation*}  
So by integrating one obtains
   \begin{equation*}
       (S_D/n)^nr^n\leq m(x_0,r)\leq m_{\kappa}(r)\qquad (0<r<r_0),
   \end{equation*}  
where the right-hand side is $O(r^n)$ as $r\rightarrow 0+$, and 
   \begin{equation*}
       S_D(S_D/n)^{n-1} r^{n-1}\leq \sigma (x_0,r)\leq \sigma_{\kappa}(r)\qquad (0<r<r_0),
   \end{equation*}  
where $\sigma_\kappa (r)=O(r^{n-1})$ as $r\rightarrow 0+$.

We now consider the behaviour at $r=0$. We obtain the bounds
  \begin{equation*}
      \frac{S_D}{m_\kappa (r)^{1/n}}
        \leq \frac{\sigma (x_0,r)}{m(x_0,r)}
        \leq \frac{n^n\sigma_\kappa (r)}{S_D^nr^n}
                 \qquad (0<r<r_0).
   \end{equation*}
By \cite[(1.18)]{CheegerGT:1982}, there exist constants $c_1(n), c_2(n)>0$ such that
   \begin{equation*}
   0 < \frac{c_1(n)}{r} 
    \leq \frac{({d}/{dr})\sigma(x_0,r)}{\sigma (x_0,r)}
    \leq \frac{c_2(n)}{r}
               \qquad (0<r<r_0).
   \end{equation*}
The exponential map is injective on the tangent space above the ball 
$B(x, r_0)$, so we can express $\sigma (x_0, r)/r^{n-1}$ and $m(x_0,r)/r^n$ as $r\rightarrow 0+$ in terms of the metric tensor and the exponentials of tangent vectors as in \cite[p. 82]{Chavel:1993}. This local expansion gives $q(r)$ as 
$r\rightarrow 0+$, thus verifying that (i) of Definition~\ref{Def4.1} is satisfied.

For $r>\delta$ we have $\sigma (x_0,r)\geq I_{\infty, \delta}({\cal M})m(x_0,r)$, and so
  \begin{equation}\label{Eq6.20}
     m(x_0,r)\geq m(x_0,\delta )\exp \bigl((r-\delta )I_{\infty , \delta }({\cal M})\bigr)
       \qquad (r>\delta )
  \end{equation}
by a direct integration. Hence  
   \begin{align}\label{Eq6.21}
   \sigma (x_0,r) & \geq I_{\infty, \delta}m(x_0,r)   \notag\\
          &\geq I_{\infty, \delta}m(x_0,\delta )\exp \bigl((r-\delta )I_{\infty, \delta }({\cal M})\bigr)
             \qquad (r>\delta ).
   \end{align}

Since $\log m(x_0,r)$ is concave, $\sigma (x_0,r)/m(x_0,r)$ decreases with $r$,
and, by (\ref{Eq6.20}), $\sigma (x_0,r)/m(x_0,r)\rightarrow 2\omega_0$ as $r \to \infty$, where $2\omega_0\geq I_{\infty, \delta}({\cal M})>0$. 
This proves conditions (ii) and (iii) for $m(x_0,r)$.

Since $\log \sigma (x_0,r)$ is concave, $\sigma' (x_0,r)/\sigma (x_0,r)$ decreases with $r$. By (\ref{Eq6.21})
$\sigma (x_0,r)\rightarrow\infty$ as $r\rightarrow\infty$. Hence 
$\sigma'(x_0, r)\geq 0$ for all $r>0$, so $\sigma (x_0, r)$ increases to infinity as $r\rightarrow\infty$. Also, $\sigma' (x_0, r)/\sigma (x_0, r)\rightarrow 2\omega_1$ where $2\omega_1\geq I_{\infty, \delta}({\cal M})>0$ by 
(\ref{Eq6.21}). Since $m'=\sigma$, we deduce that $\omega_1=\omega_0$.
\end{proof}

%%%%%%%%%%%%%%%%%%%%%%%%%%%%%%%%%%%%%%%%%%%%%%%%%%%%%%%%%%%%%%%%%%%%%%%%%%%%%%%%%%%%%%%%%%%%%%%%%%

\end{document}